\documentclass[a4paper,11pt,twoside,DIV=calc]{scrartcl}	

\usepackage{tocenter}				
\ToCenter[h]{\textwidth}{\textheight}	

\setkomafont{section}{\bfseries\Large}		

\usepackage{scrpage2}      	
\ofoot{}					
\ifoot{}
\rehead{\pagemark}			
\cehead{OLIVIER GABRIEL AND MARTIN GRENSING}	
\rohead{\pagemark}			
\cohead{SMOOTH SUBALGEBRAS OF PIMSNER ALGEBRAS}

\input xy
\xyoption{all}
\usepackage{lscape}
\usepackage[latin1]{inputenc}
\usepackage[cmtip,arrow]{xy}
\usepackage[leqno]{amsmath}
\usepackage{amssymb,amscd,amsthm,mathrsfs,yfonts,manfnt,pb-diagram,pb-xy}
\usepackage[nice]{nicefrac}		
\usepackage[T1]{fontenc}

\usepackage{ifthen}
\newboolean{comm}
\setboolean{comm}{true}
\newboolean{sol}
\setboolean{sol}{false}
\newboolean{notes}
\setboolean{notes}{true}			
\CompileMatrices  	

\usepackage{hyperref}

\newcommand{\com}{\ifthenelse{\boolean{comm}}}
\newcommand{\sol}{\ifthenelse{\boolean{sol}}}
\newcommand{\note}{\ifthenelse{\boolean{notes}}}

\newtheorem{Def}{Definition}
\newtheorem{Not}[Def]{Notation}
\newtheorem{Prop}[Def]{Proposition}

\newtheorem{Th}[Def]{Theorem}

\newtheorem{Lem}[Def]{Lemma}

\theoremstyle{Def}
\newtheorem{Rem}[Def]{Remark}

\newcommand{\mC}{\ensuremath{\mathbb{C}}}					
\newcommand{\mN}{\ensuremath{\mathbb{N}}}
\newcommand{\mZ}{\ensuremath{\mathbb{Z}}}
\newcommand{\mK}{\ensuremath{\mathbb{K}}}
\newcommand{\mT}{\ensuremath{\mathbb{T}}}

\newcommand{\mc}{\mathcal}								

\DeclareMathOperator{\Image}{Im}

\newcommand{\ee}{\mbox{$\varepsilon$}}

\renewcommand{\phi}{\varphi}

\newcommand{\mb}{\begin{pmatrix}}					
\newcommand{\me}{\end{pmatrix}}						
\newcommand{\stoe}{\ensuremath{\mathcal{T}^\infty}}		

\DeclareMathOperator{\Ad}{Ad}							
\newcommand{\sco}{\ensuremath{\mathcal{K}}}					

\DeclareMathOperator{\Ker}{Ker}						

\DeclareMathOperator{\id}{id}							

\newcommand{\rrarrow}{\rightrightarrows}

\usepackage[T1]{fontenc}
\usepackage{lmodern}
\usepackage{graphicx}
\usepackage[english]{babel}
\usepackage{amsfonts}
\usepackage{amssymb,amsmath,amsthm,amscd}
\usepackage{mathrsfs}
\usepackage{slashed}
\usepackage{authblk}

\usepackage{enumerate} 
\usepackage{xspace} 
\usepackage{xargs}

\newcommand{\mathbbm}{\mathbb}

\newcommand{\Aa}{\mathscr{A}}

\newcommand{\action}{\curvearrowright}

\newcommand{\derp}[2][]{\frac{\partial #1}{\partial #2}}

\newcommand{\inj}{\hookrightarrow}

\newcommand{\Modl}{\mathcal{E}}
					\newcommandx*{\ModQHM}[2][1={},2={}]{{M^{#1}_{#2}}}

\newcommand{\N}{\mathbbm{N}}
							\newcommandx*\QHM[2][1={},2={}]{{D^{#1}_{#2}}}
							\newcommandx*\QHMl[2][1={},2={}]{{\mathscr{D}^{#1}_{#2}}}
\newcommand{\R}{\mathbbm{R}}

\newcommand{\Toep}{\stoe}
\newcommand{\ToepTop}{{\mathcal{T}_\mc{B}}}
\newcommand{\ToepTopAlg}{{\mathcal{T}_\mc{B}^0}}

\newcommand{\Tt}{\mathcal{T}}

\newcommand{\TensTop}{\otimes_\pi}

\newcommand{\Z}{\mathbbm{Z}}

\DeclareMathOperator{\im}{Im}

\setboolean{notes}{true}

\renewcommand{\sco}{\ensuremath{\mathbb{K}^\infty}}

\begin{document}

\title{\Large{\MakeUppercase{Six-Term Exact Sequences for Smooth Generalized Crossed Products}}}
\author{\small{Olivier \textsc{GABRIEL} and Martin \textsc{GRENSING}}}
\affil{%
Université Denis \textsc{Diderot} -- Paris 7\\
175, rue du Chevaleret, 75 013 Paris FRANCE\\
E-mail: \texttt{gabriel@math.jussieu.fr} and \texttt{mgren@math.jussieu.fr}}

\date{\small{\today}}
\maketitle
\thispagestyle{empty}
\begin{abstract} \small{We define smooth generalized crossed products and prove six-term exact sequences of Pimsner-Voiculescu type. This sequence may, in particular, be applied to smooth subalgebras of the Quantum Heisenberg Manifolds in order to compute the generators of their cyclic cohomology. Our proof is based on a combination of arguments from the setting of (Cuntz-)Pimsner algebras and the Toeplitz proof of Bott-periodicity.}
\end{abstract}

\tableofcontents

\section{Introduction} In classical bivariant $K$-theory, the Pimsner-Voiculescu-six-term exact sequence (PMV-sequence)  \cite{MR587369} is one of the main results, allowing to calculate the $K$-theory of crossed products in many  cases. While proved originally for monovariant $K$-theory, it can be proved more simply using $KK$-theory, see \cite{BlackK}. J. Cuntz (\cite{MR750677}) has given a proof that applies, at least partially, to all split exact, homotopy invariant stable functors on the category of $C^*$-algebras with values in the category of abelian groups. His proof is based on the notion of quasihomomorphism (see Definition \ref{Def:Qmorphisme}). These techniques were later generalized to the setting of locally convex (or even bornological) algebras (\cite{CuntzMeyer}). The by now classical Toeplitz proof of Bott periodicity for $C^*$-algebras (\cite{MR750677})  may then be modified to yield a proof of the six-term exact sequence, and J. Cuntz gives in \cite{MR2240217} an analogous sequence involving smooth crossed products in the locally convex setting, and outlines a proof.

Our result partially extends the results obtained by R. Nest in \cite{CohomCyclZ}. In this article, a six-term exact sequence was established for smooth crossed products by $\Z$, for the particular case of periodic cyclic cohomology and under additional hypotheses on the algebra and the action.

Meanwhile, M. Pimsner presents in \cite{MR1426840} a simultaneous generalization of crossed products and Cuntz-Krieger algebras, and proves a PMV-type sequence, using a certain universal Toeplitz-type $C^*$-algebra. The algebra we are ultimately interested in is a ``smooth version'' of the so-called Pimsner algebras, which can be described (at least in the $C^*$-setting) in an essentially equivalent way using generalized crossed products (see \cite{AbadieEE} and section \ref{Sec:SmoothCPAlg} below). We would like to note that, eventhough we do not refer to \cite{KhoshSkand} explicitly, it has to some extent guided our approach.

Our aim is to establish a six-term exact sequence for half-exact, diffotopy-invariant and $\sco$-stable functors. The proof is a  combination of Pimsner's methods in the $C^*$-setting and Cuntz's techniques introduced in the Toeplitz proof of Bott periodicity in the smooth setting.  One of the main results is the ``equivalence'' of a certain Toeplitz-type algebra with a smaller algebra as formulated in Theorem \ref{PMV}. Further, the passage from crossed products to our setting naturally causes difficulties when considering the equivalence between the ``base algebra'' and a certain ideal inside the Toeplitz-type algebra. It is here that we have to impose further size-restrictions on the algebras we start with. We finally apply this to Quantum Heisenberg Manifolds (QHM).

The contents of this article are as follow: in section \ref{Sec:Prelim}, we introduce some general theory related to locally convex algebras. Section \ref{contexts} is devoted to a smooth version of Morita equivalence, namely Morita contexts. Quasihomomorphisms and their properties are introduced in section \ref{Sec:Qmorphisme}, while the smooth Toeplitz algebra (in the sense of Cuntz) is defined in section \ref{Sec:stoe}. Both will play essential parts in our construction. 

At this point, we are ready to introduce our notion of \emph{smooth generalized crossed product} (section \ref{Sec:SmoothCPAlg}) and the associated Toeplitz-type algebras (section \ref{Sec:ExtAlgLC}). A natural extension of locally convex algebras is then given, which involves a kernel and the Toeplitz-type algebra. To get a PMV-sequence, we need to prove that both the kernel and the Toeplitz algebra are ``equivalent'' to a certain base algebra (compare Theorems \ref{context} and \ref{PMV}). This is taken care of in sections \ref{MoritaContext} and \ref{Sec:ToepBase}, respectively. Finally, we assume half-exactness of $H$ and obtain the six-term exact sequence in section \ref{Sec:6terms}. 

As applications, we first consider smooth crossed products in section \ref{Sec:AppliPiCrG}; then in section \ref{Sec:AppliQHM}, we construct  six-term exact sequences for the Quantum Heisenberg Manifolds (QHM) in $kk$ and cyclic theory, that allow to completely determine the cyclic invariants of the QHM. The $C^*$-QHM can be considered as generalized crossed products, and therefore the long exact sequences are available. When dealing with cyclic theory, one has to stay in the smooth category. We show thus that the smooth QHM satisfy our growth conditions, and our techniques therefore apply. Consequently, we get Theorem \ref{QHMDIAG} below as a concrete result. 

The explicit calculation of the periodic cyclic cohomology of the QHM, based on our results, is done in another article (see \cite{OlPair}).

Both authors would like to express their gratitude to G. Skandalis for helpful discussions, remarks and advice.

\section{Preliminaries}
\label{Sec:Prelim}
We denote the completed projective tensor product by $\otimes_\pi$. By a \emph{locally convex algebra} we will mean a complete locally convex vector space that is at the same time a topological algebra, \textsl{i.e.}, the multiplication is continuous. Hence if $\mc{A}$ is a locally convex algebra, this means that for every continuous seminorm $p$ on $\mc{A}$ there is a continuous seminorm $q$ on $\mc{A}$ such that for all $a,b\in \mc{A}$
$$p(ab) \leqslant q(a)q(b).$$
A \emph{$m$-algebra} is a locally convex algebra whose topology is defined by submultiplicative seminorms, \textsl{i.e.}, we can take $q = p$ in the above inequality.

If $\mc{A}$ is a locally convex algebra, we denote by $Z\mc{A}$ the locally convex algebra of differentiable functions from $[0,1]$ to $\mc{A}$ all of whose derivatives vanish at the endpoints. We denote by $ev^\mc{A}_t \colon  Z\mc{A}\to \mc{A}$ the evaluations in $t\in I = [0,1]$. The \emph{smooth suspension} $\mathscr{S} \mc{A}$ of a locally convex algebra $\mc{A}$ is 
$$\mathscr{S} \mc{A} = \left\{ f \in C^\infty([0,1], \mc{A}) \middle| \forall k \in \N, f^{(k)}(0) = f^{(k)}(1) = 0 \right\} .$$
Notice that $\mathscr{S} \mc{A}$ is \emph{not} equal to $Z \mc{A}$ since for $\mathscr{S} \mc{A}$ we ask that $f(0) = f(1) = 0$.

\begin{Def}
Let $\phi_0,\phi_1\colon \mc{A} \to \mc{B}$ be homomorphisms of locally convex algebras. A \emph{diffotopy} between $\phi_0$ and $\phi_1$ is by definition a homomorphism $\Phi \colon \mc{A}\to Z \mc{B}$ such that $ev^\mc{B}_i\circ\Phi=\phi_i$ for $i=0,1$.
\end{Def}

\begin{Def} The \emph{smooth compact operators} are defined as those  compact operators $A\in\mathbb{B}\left (\ell^2({\mN}) \right )$ such that, if $(a_{i,j})$ is the representation of $A$ with respect to the standard basis, then for all $m,n\in\mN$:
$$\sup_{i,j\in\mN} (1+i)^n(1+j)^m|a_{i,j}|\leqslant \infty.$$
They are topologized by the increasing family of seminorms $\|_{m,n}$ with
$$\|A\|_{m,n}:=\sum_{i,j} (1+i)^m(1+j)^n|a_{i,j}|.$$
\end{Def}
For any locally convex algebra $\mc{B}$, the elements of $\sco\otimes_\pi \mc{B}$ are just the matrices with rapidly decreasing coefficients in $\mc{B}$ (see \cite{CuntzMeyer}, chapter 2, 3.4).

\begin{Def}
We distinguish between on the one hand \emph{split exact sequences} in the category of locally convex algebras, \textsl{i.e.},
\[\xymatrix{0\ar[r]&\mc{A}\ar[r]&\mc{B}\ar[r]&\mc{C}\ar[r]&0}\]
with a splitting $\mc{C} \to \mc{B}$ which is an \emph{algebra homomorphism} and on the other hand \emph{linearly split exact sequences}, in which the splitting is only a (continuous) linear map.
\end{Def}

\begin{Def}
A functor $H$ on the category of locally convex algebras with values in abelian groups is called 
\begin{itemize}
\item \emph{$\sco$-stable} if the natural inclusion $\theta:\mc{A}\to \sco\otimes_\pi \mc{A},\; a\mapsto e\otimes a$ obtained by any minimal idempotent $e\in \sco$ induces an isomorphism
$$H(\theta) \colon  H(\mc{A})\to H(\sco\otimes_\pi \mc{A}),$$
\item \emph{diffotopy-invariant}, if $H(ev^{\mc{A}}_0)=H(ev^{\mc{A}}_1)$  for every locally convex algebra $\mc{A}$,
\item \emph{split exact} if for every short exact sequence of locally convex algebras
\[\xymatrix{0\ar[r]&\mc{A}\ar[r]&\mc{B}\ar[r]&\mc{C}\ar[r]&0}\]
admitting a split  the sequence
\[\xymatrix{0\ar[r]&H(\mc{A})\ar[r]&H(\mc{B})\ar[r]&H(\mc{C})\ar[r]&0}\]
is exact.
\end{itemize}
\end{Def}

\section{Morita contexts and split exact functors}
\label{contexts}

The following definition of Morita context is basically from \cite{MR2240217}.  We add some isomorphisms to the definitions in \cite{MR2240217} in order to make the existence of a Morita context a weaker condition than being isomorphic.

\begin{Def}
\label{contextcond}
Let $\mc{A}$ and $\mc{B}$ be locally convex algebras. A \emph{Morita context} from $\mc{A}$ to $\mc{B}$ is given by data $(\phi,\mc{D},\psi,\xi_i,\eta_i)$, where $\mc{D}$ is a locally convex algebra, $\phi \colon \mc{A}\to \mc{D}$, $\psi \colon \mc{B}\to \mc{D}$ are isomorphisms onto subalgebras of $\mc{D}$, and sequences $\eta_i$, $\xi_i$ in $\mc{D}$ such that 
\begin{enumerate}
\item $\eta_i \phi(\mc{A})\xi_j\in\psi(\mc{B})$ for all $i,j$;
\item $(\eta_i \phi(a)\xi_j)_{ij}\in\sco\otimes_\pi \psi(\mc{B})$;
\item\label{convergencecontext} $\sum \xi_i\eta_i \phi(a)=\phi(a)$ for all $a\in \mc{A}$ (convergence in $\phi(\mc{A})$).
\end{enumerate}
\end{Def}

With this definition:
\begin{itemize}
\item If $\phi:\mc{A}\to \mc{B}$ is an isomorphism, then we get a Morita context $(\phi^+,\mc{B}^+,\cdot^+)$ from $\mc{A}$ to $\mc{B}$, where $\mc{B}^+$ is the unitization of $\mc{B}$, $b\mapsto b^+$ the canonical embedding, and $\phi^+=\phi\circ\cdot^+$.
\item In particular, there is now a canonical Morita context from $\mc{A}$ to $\mc{A}$, for any locally convex algebra $\mc{A}$.
\item If $\mc{B}\subseteq \sco\otimes_\pi \mc{C}$ is a subalgebra, and we call a corner in $\mc{B}$ a subalgebra $\mc{A}\subseteq \mc{B}$ of the form $\sum_{i=0}^k e_{0i}\mc{B}\sum_{i=0}^k e_{i0}$, then there is a trivial context from $\mc{A}$ to $\mc{B}$.
\item If $\mc{B}$ is row and column-stable ($e_{0i}\mc{B},\mc{B}e_{i0}\subseteq \mc{B}$ for all $i$), then there is a context from $\mc{B}$ to $\mc{A}$, where $\mc{B}$ is a subalgebra of $\sco\otimes_\pi \mc{A}$.
\end{itemize}
 
\begin{Def} 
If $H$ is a $\sco$-stable functor, then a Morita context from $\mc{A}$ to $\mc{B}$ induces a map $\theta \colon \mc{A} \to \sco\otimes_\pi \mc{B},\,a\mapsto (\psi^{-1}(\eta_i\phi(a)\xi_j))_{ij}$ and a morphism $H(\mc{A}) \to H(\mc{B})$ of abelian groups $H(\phi,\mc{D},\psi,\xi_i,\eta_i):=H(\theta)$.
\end{Def}
Note that for $a,a'\in \mc{A}$ we have 
\begin{align*}
\left(\psi^{-1}(\eta_i\phi(a)\xi_j) \right)_{ij} \left(\psi^{-1}(\eta_k\phi(a')\xi_l)\right)_{kl}= \left(\psi^{-1}\left(\eta_i\phi(a)\sum_{m}\xi_m\eta_m \phi(a')\xi_j \right)\right)_{ij}
\end{align*}
which equals $\theta(aa')$ by \ref{convergencecontext} in the definition of a Morita context. Hence $\theta$ is indeed a homomorphism.

\begin{Def}
\label{Def:BicontextMorita}
 A \emph{Morita bicontext} between locally convex algebras $\mc{A}$ and $\mc{B}$ is given by two Morita contexts from $\mc{A}$ to $\mc{B}$ and $\mc{B}$ to $\mc{A}$ respectively, of the form $(\phi, \mc{D},\psi,\xi_i^\mc{A},\eta_i^\mc{A})$ and $(\psi,\mc{D},\phi,\xi_i^\mc{B},\eta_i^\mc{B})$ such that 
\begin{enumerate}
\item $\phi(\mc{A})\xi_i^\mc{A}\xi_j^\mc{B}\subseteq\phi(\mc{A})$, $\eta_i^\mc{B}\eta_j^\mc{A}\phi(\mc{A})\subseteq \phi(\mc{A})$ (left compatibility);
\item $\psi\mc{B}\xi_i^\mc{B}\xi_j^\mc{A}\subseteq \psi\mc{B}$, $\eta_i^\mc{A}\eta_j^\mc{B}\psi\mc{B}\subseteq \psi\mc{B}$ (right compatibility).
\end{enumerate}
\end{Def}

\begin{Th}\label{contweakeriso} Given two Morita contexts as in the above definition, and a diffotopy-invariant, $\sco$-stable functor $H$, we have
$$H(\psi,\mc{D},\phi,\xi_i^\mc{B},\eta_i^\mc{B})\circ H(\phi, \mc{D},\psi,\xi_i^\mc{A},\eta_i^\mc{A})= \id_{H(\mc{A})}\text{ if they are left compatible;}$$
$$ H(\phi, \mc{D},\psi,\xi_i^\mc{A},\eta_i^\mc{A})\circ H(\psi,\mc{D},\phi,\xi_i^\mc{B},\eta_i^\mc{B})=\id_{H(\mc{B})}\text{ if they are right compatible.}$$
\end{Th}\newcommand{\potimes}{\otimes_\pi}
\begin{proof}
The proof is an adaptation of the one from \cite{MR2240217}, Lemma 7.2. Denote the isomorphism $H(\mc{A})\to H(\sco\potimes \mc{A})$ given by $\sco$-stability by $\ee_\mc{A}$. Then more precisely, we have to show that
\begin{align*}
H(\phi,\mc{D},\psi,\xi_i^\mc{B},\eta_i^\mc{B})\circ \ee_\mc{B}^{-1}\circ H(\phi, \mc{D},\psi,\xi_i^\mc{A},\eta_i^\mc{A})
\end{align*}
is invertible. We suppose left compatibility; then denoting $\theta^\mc{A}$ and $\theta^\mc{B}$ the maps $\mc{A}\to \sco\potimes \mc{B}$ and $\mc{B}\to\sco\potimes \mc{A}$ determined by the two contexts, and multiplying by $\ee_{\sco\potimes \mc{A}}$  on the left, we see that it suffices to show that the composition $(\sco\otimes\theta_\mc{B})\circ\theta_\mc{A}$ induces an invertible map under $H$. Now this is the map
\begin{equation}
\label{compo}
a\mapsto (\phi^{-1}(\eta_i^\mc{B}(\eta_k^\mc{A}\phi(a)\xi_l^\mc{A})\xi_j^\mc{B}))_{iklj}
\end{equation}
and it is diffotopic to the stabilisation as follows. The $L\times L$ matrix with entries $\phi^{-1}(\hat\eta_\alpha(t)\phi(a)\hat\xi_\beta(t))$, $\alpha,\beta\in \mN^2\cup\{0\}$ with
\begin{align*}
\hat\xi_0(t)=\cos(t) 1\hspace{2cm} \hat\xi_{il}(t)=\sin(t)\xi_i^B\xi_l^A\\
\hat\eta_0(t)=\cos(t) 1\hspace{2cm}\hat\eta_{kj}=\sin(t) \eta_k^B\eta_j^A
\end{align*}
yields a diffotopy of the map in equation \eqref{compo}.
\end{proof}

\section{Quasihomomorphisms}
\label{Sec:Qmorphisme}
We collect some basic properties of quasihomomorphisms as introduced first in \cite{MR733641}, and  later adapted to the locally convex setting. See for example \cite{MR2207702} and \cite{MR2240217} for some details.

\begin{Def} 
\label{Def:Qmorphisme}
A \emph{quasihomomorphism} between locally convex algebras $\mc{A}$ and $\mc{B}$ relative to $\hat{\mc{B}}$, denoted $(\alpha,\bar\alpha) \colon  \mc{A}\rrarrow \hat{\mc{B}}\trianglerighteq \mc{B}$, is given by a pair $(\alpha,\bar\alpha)$ of homomorphisms into a locally convex algebra $\hat{\mc{B}}$ that contains $\mc{B}$ such that $(\alpha-\bar\alpha)(\mc{A})\subseteq \mc{B}$, $\alpha(\mc{A})\mc{B}\subseteq \mc{B}$ and $\mc{B}\alpha(\mc{A})\subseteq \mc{B}$. 
\end{Def}

Note that the definition is symmetric in $\alpha$ and $\bar\alpha$. Also, the definition can be generalized, basically by supposing that $\alpha,\bar\alpha$ have image in the multipliers of $\mc{B}$, but we will not go into this. That quasihomomorphisms may be used as a tool to prove properties of $K$-theory goes back to \cite{MR750677}.
\begin{Prop}
\label{quasiprops}
Let $(\alpha,\bar\alpha):\mc{A}\rrarrow \hat{\mc{B}}\trianglerighteq \mc{B}$ be a quasihomomorphism, $H$ a split exact functor with values in the category of abelian groups. We have the properties
\begin{enumerate}
\item $(\alpha,\bar\alpha)$ induces a morphism, denoted $H(\alpha,\bar\alpha)$, from $H(\mc{A})$ to $H(\mc{B})$;
\item for every morphism $\phi \colon  \mc{D}\to \mc{A}$, $(\alpha\circ\phi,\bar\alpha\circ\phi)$ is a quasihomomorphism and:
$$H(\alpha,\bar\alpha)\circ H(\phi)=H(\alpha\circ\phi,\bar\alpha\circ\phi);$$
\item if $\phi:\hat{\mc{B}}\to \hat{\mc{C}}$ is a morphism and $\mc{C}\trianglelefteq\hat{\mc{C}}$ a subalgebra such that $(\phi\circ\alpha,\phi\circ\bar\alpha) \colon  \mc{A}\rrarrow\hat{\mc{C}}\trianglerighteq \mc{C}$ is a quasihomomorphism, then $H(\phi\circ\alpha,\phi\circ\bar\alpha)=H(\phi')\circ H(\alpha,\bar\alpha)$, where $\varphi'$ is the restriction of $\varphi$ to a morphism $\mc{B} \to  \mc{C}$;
\item \label{Pt:4} if $\mc{A}$ is generated as a locally convex algebra by a set $X$, $\mc{B}\trianglelefteq \hat{\mc{B}}$ a subalgebra and $\alpha,\bar\alpha \colon  \mc{A}\to \hat{\mc{B}}$ are two homomorphisms such that $(\alpha-\bar\alpha)(X)\subseteq \mc{B}$, then $(\alpha,\bar\alpha)$ is a quasihomomorphism $\mc{A}\rrarrow\hat{\mc{B}}\trianglerighteq \mc{B}$;
\item if $\alpha-\bar\alpha$ is a morphism and orthogonal to $\bar\alpha$, then $H(\alpha,\bar\alpha)=H(\alpha-\bar\alpha)$;
\item if $H$ is diffotopy-invariant, $(\alpha,\bar\alpha)\colon  \mc{A}\rrarrow Z\hat{\mc{B}}\trianglerighteq Z\mc{B}$  a quasihomomorphism, then the homomorphisms $H(ev_t\circ\alpha,ev_t\circ\bar\alpha)$ is independent of $t\in I$. 
\end{enumerate}
\end{Prop}

\begin{proof}
To prove (iv), we let $\mc{A}$ be generated by $X$, and $x,x'\in X$; then $$\alpha(xx')-\bar\alpha(xx')=(\alpha(x)-\bar\alpha(x'))\alpha(x')+\bar\alpha(x')(\alpha(x')-\bar\alpha(x')).$$ 
The proof of the other well-known statements can be found for example in \cite{GrensingUniversal}, Proposition 23.
\end{proof}
For subsets $X$ and $Y$ in a locally convex algebra $\mc{B}$, we say $Y$ is $X$-stable if $XY,YX\subseteq Y$.
\begin{Def} 
Let $X,Y\subseteq \mc{B}$ be subsets. Then $XYX$ is defined as the smallest closed $X$-stable (hence locally convex) subalgebra of $\mc{B}$ containing $Y$.
\end{Def}
$XYX$ is obviously independent of the size (but not the topology) of the ambient algebra: if $\mc{B}\subseteq \mc{B}'$ as a closed subalgebra and $X,Y\subseteq \mc{B}$, then it doesn't matter if we intersect $X$-stable subalgebras of $\mc{B}$ or $\mc{B}'$. We denote by $X^+$ the set $X$ with a unit adjoint. We have the following obvious

\begin{Lem}
\label{triviallem} 
$XYX$ is the locally convex algebra generated by 
$$\left\{\sum_{finite}x_iy_ix_i' \middle|x_i,x_i'\in X^+,\;y_i\in Y\right\}$$
and thus depends only on the locally convex algebras $LC(X)$ and $LC(Y)$ generated in $B$ by $Y$ and $X$, respectively:
$$XYX=X LC(Y)X=LC(X)YLC(X).$$
\end{Lem} 

\begin{Def}
For a pair of morphisms $\alpha,\bar\alpha \colon  \mc{A}\to \hat{\mc{B}}$ we call the quasihomomorphism $(\alpha,\bar\alpha)  \colon 	\mc{A}\rrarrow \hat{\mc{B}}\trianglerighteq \mc{B}$ with $\mc{B}:=\alpha(\mc{A})(\alpha-\bar\alpha)(\mc{A})\alpha(\mc{A})$ the \emph{associated quasihomomorphism}. We call a quasihomomorphism \emph{minimal}, if it is associated to $\alpha,\bar\alpha \colon  \mc{A}\to \hat{\mc{B}}$.
\end{Def}
By the above Lemma, if $X$ generates $\mc{A}$, then $\alpha(X)(\alpha-\bar\alpha)(X)\alpha(X)=\alpha(\mc{A})(\alpha-\bar\alpha)(\mc{A})\alpha(\mc{A})$, using the notation of Lemma \ref{triviallem} above.

\section{The smooth Toeplitz algebra}
\label{Sec:stoe}

We recall the definition and properties of the smooth Toeplitz algebra as introduced in \cite{MR1456322}, compare Satz 6.1 therein:
\begin{Def}
\label{Def:ToepLisse}
 The \emph{smooth Toeplitz algebra} $\stoe$ is defined as the direct sum $\sco\oplus C^\infty(S^1)$ of locally convex vector spaces with multiplication induced from the inclusion into the $C^*$-Toeplitz algebra.
\end{Def}
It follows from \cite{MR1456322} that $\mc{T}^\infty$ is a nuclear locally convex algebra. We denote $S$ and $\bar S$ the generators of $\stoe$, set $e:=1-S\bar S$, and deduce the following fact from the proof of Lemma 6.2 in \cite{MR1456322}:

\begin{Lem}
\label{ahomotopy}
There exists a unital diffotopy $\phi_t$ such that
$$ \phi_t(S)=S(1- e)\otimes 1+f(t)(e\otimes S)+g(t)Se\otimes 1,$$
where $f,g$ are smooth functions such that $f(0)=0$, $f(1)=1$, $g(0)=1$ and $g(1)=0$ and all derivatives of $f$ and $g$ vanish in $0,1$.
\end{Lem}

Thus $\phi_0:\stoe \to \stoe\otimes_\pi\stoe,\,x\mapsto x\otimes 1$ is the canonical inclusion into the first variable, and $\phi_1$ is the map determined by $S\mapsto \phi_0(S^2\bar S)+(1-S\bar S)\otimes S$.

To unclutter notation, we set $\phi_0(x)= x \otimes 1 =:\hat x$ and $e\otimes x=:\check x$ for all $x\in\stoe$. We also recall that that the smooth Toeplitz algebra fits into an extension 
$$\xymatrix{ 0\ar[r] &\sco\ar[r]& \stoe\ar[r]^-\pi& LC(U)\ar[r]&0}$$
where $LC(U)$ denotes the algebra $C^\infty(\mT)$ of smooth functions on the torus. The extension is linearly split by a continuous map $\rho:LC(U)\to \stoe$. The ideal is the  algebra $\sco$ of smooth compact operators, which is isomorphic (as a vector space) to $s\otimes s$, where $s$ is the space of rapidly decreasing sequences, and hence $\sco$ is nuclear. We caution the reader that the smooth compacts are not invariant under conjugation by unitaries if we view them as represented inside some $\mK$, where $\mK$ denotes the compacts on a separable, infinite dimensional Hilbert space. They are rapidly decreasing with respect to a \underline{choice} of base.

\section{Smooth generalized crossed products} 
\label{Sec:SmoothCPAlg}

In the following, a \emph{gauge action} $\gamma$ on a locally convex algebra $\mc{B}$ is an action of $S^1$ which is pointwise continuous. In this case, we denote by $\mc{B}^{(m)} $ for $m \in \Z$ the space of elements of \emph{degree} $m$, \textsl{i.e.},
$$ \mc{B}^{(m)} = \left\{ b \in \mc{B} \middle| \forall t \in S^1, \gamma_t(b) = e^{i 2 \pi m t} b \right\} .$$
We denote the degree $0$ part by $\mc{A}$ and the degree $1$ part by $\mc{E}$.

An element $b$ of $\mc{B}$ is called \emph{gauge smooth} if the map $t \mapsto \gamma_t(b)$ is smooth.

\begin{Def}
\label{Def:SPimsner}
A \emph{smooth Generalized Crossed Product} (GCP) $\mc{B}$ is a locally convex algebra equipped with an involution and a gauge action such that 
\begin{itemize}
\item
$\mc{B}^{(0)}$ and $\mc{B}^{(1)}$ generate $\mc{B}$ as a locally convex algebra;
\item
all $b \in \mc{B}$ are gauge smooth, and the induced map $\mc{B} \to C^\infty(S^1 \to \mc{B})$ is continuous.
\end{itemize}
\end{Def}

In the following, we sketch how to compare this definition with the $C^*$-case: our definition here focusses on the gauge action, thereby remains close to the notion of generalized crossed product for $C^*$-algebras as given by Abadie, Eilers and Exel in \cite{AbadieEE}. There is however a way to relate generalized crossed products to the so called Pimsner algebras (see \cite{MR1426840}) by using the work of Katsura \cite{Katsura02}. The techniques to be used in our proofs are inspired by Pimsner's approach. 

The essential difference between the viewpoints of  \cite{AbadieEE} and  \cite{MR1426840} is the following: Pimsner's construction focusses on the universal properties of the algebra, while  \cite{AbadieEE} emphasises the gauge action.  All Pimsner algebras have a gauge action, and taking the degree zero part for this action as base algebra results in a loss of information. More precisely, the data specified for a generalized crossed product suffices to construct a Pimsner algebra, which coincides with the generalized crossed product. If we start however with an algebra as constructed by Pimsner, taking the degree zero part as base algebra allows to represent it as a generalized crossed product (this is closely related to the change of coefficient algebra in section 3 of \cite{MR1426840}).

In a first version of this article \cite{ChernGrensingGabriel}, we had focused on Pimsner's description. The seemingly more intrinsic treatment results in technical complications and we have found that the additional complexity of the proofs was not justified by the applications we have in mind.

For Pimsner algebras as well as for generalized crossed products, the gauge action is pointwise continuous and therefore we obtain a smooth subalgebra to which our definition applies. For more details and a more general class of examples, see Remark \ref{Rem:ExGCP} below.

\begin{Rem}
\label{Rem:SPimsner}
It is easy to see that:
\begin{itemize}
\item
$\mc{A}=\mc{B}^{(0)}$ is an involutive subalgebra of $\mc{B}$ which we refer to as the \emph{base algebra} of $\mc{B}$.
\item
$\mc{E} = \mc{B}^{(1)}$ is an $\mc{A}$-bimodule. More generally, $\mc{B}^{(m)}$ is an $\mc{A}$-bimodule that we write as $\mc{E}^{(m)}$.
\item
If we denote by $\bar{\cdot }$ the involution on $\mc{B}$, the expressions $\xi \bar{\eta}$ and $\bar{\xi} \eta$ define ``$\mc{A}$-valued scalar products'' on $\mc{E}^{(m)}$.
\end{itemize}
\end{Rem}

\begin{Not}
In the following we will use the symbols $a$ and $\xi$ as typical notations for elements of $\mc{A}$ and $\mc{E}$, respectively. To distinguish between $\xi$ as element of the bimodule $\mc{E}$ and $\xi $ as element of $\mc{B}$, we denote by $S \colon \mc{E} \to \mc{B}$ the inclusion and abbreviate $S(\xi)$ by $S_\xi$.
\end{Not}

Our final result requires stronger hypotheses. To formulate it, say that a finite sequence $\eta_1^{(m)}, \ldots , \eta_n^{(m)}$ in $\Modl ^{(m)}$ is a frame for $\Modl ^{(m)}$ if
\begin{equation}
\label{Eqn:Repere}
\sum_{i =1}^n \eta_i^{(m)} \, \overline{\eta_i^{(m)}} = 1 .
\end{equation}

The condition \eqref{Eqn:Repere} implies in particular that $\mc{E}^{(m)}$ is ``full'' for the ``left scalar product'', \textsl{i.e.}, the span of all the $\xi \bar \eta$ is dense in (even equal to) $\mc{A}$. Finally, the definition of tame smooth GCP requires some notation: 
\begin{Not}
\label{Not:Xi}
If $\mc{B}$ is a smooth GCP and which has frames $(\eta_i^{(m)} )_i$ for every $\Modl  ^{(m)}$ with $0 \leqslant m$, we can sort the $\eta_i^{(m)}$ according to the lexicographical order of $(i,m)$. We then denote by $(\xi_l)_l$ the $\eta_i^{(m)}$ numbered according to this order.

For instance, if we have the sequences 
\begin{align*}
\eta^{(0)}_1 &= 1 \in \mc{A},
& 
& \left (\eta_i^{(1)} \right )_{i \in \{1, 2,3\}},
&
& \left (\eta_i^{(2)}\right )_{i \in \{1, 2,3\}},
&
&\ldots
\end{align*}
 then the sequence $(\xi_l)$ will be:
\begin{align*}
\xi_1 &= 1 \in \mc{A}
&
\xi_2 &= \eta_1^{(1)}
&
\xi_3 &= \eta_2^{(1)}
&
\xi_4 &= \eta_3^{(1)}
&
\xi_5 &= \eta_1^{(2)}
&
\xi_6 &= \eta_2^{(2)}
&
&\ldots 
\end{align*}
Notice that with this procedure, $(\xi_l)$ is a sequence which takes its values in $\bigcup_{ m \in \N} \mc{E}^{(m)}$ and not just in $\mc{E} $. 
 \end{Not}

\begin{Not}
\label{Not:MultiVect}
In order to simplify upcoming calculations we extend the definition of $S_\xi$ as follows: if $I\in \N^{n}$ is a multiindex, then we denote by $\xi_I$ the $n$-tuple $(\xi_{I_1},\ldots,\xi_{I_n})$ and set $S_{\xi_I}=S_{\xi_{I_1}}\cdots S_{\xi_{I_n}}$, $\bar S_{\xi_I} = \bar S_{\xi_{I_n}}\cdots \bar S_{\xi_{I_1}}$. We will refer to $\xi_I$ as an \emph{$\mc{E}$-multivector} and write $|\xi_I|:=|I|:=n$.
\end{Not}

\begin{Def}
\label{Def:PiCrGTame}
A \emph{tame smooth GCP} is a smooth GCP admitting frames in all degrees that satisfy
\begin{enumerate}[(i)]
\item \label{Pt:PiCrGTame1}
there is an integer $N$ such that for any $m \in \N$, the finite sequence $(\eta_i^m)_i$ has at least $1$ and at most $N$ elements;
\item \label{Pt:PiCrGTame2}
for any continuous seminorm $p$ on $\mc{B}$, the (real) sequences $\big( p\left( S_{\xi_l} \right) \big )_l $ and $\big ( p( \bar S_{\xi_l} ) \big )_l$ have polynomial growth in $l$.
\end{enumerate}
\end{Def}

\begin{Rem}
\label{Rem:CroissMod}
Two consequences of being tame are:
\begin{itemize}
\item
for any $l \in \N$, $l/N \leqslant |\xi_l| \leqslant l$;
\item 
for any continuous seminorm $q$ on $\Toep $, the (real) sequences $\left( q  \left (S^{|\xi_l|} \right) \right)_l$ and $\left( q \left( \bar S^{|\xi_l|} \right) \right)_l$ have polynomial growth in $l$.
\end{itemize}
\end{Rem}

\section{An extension of locally convex algebras}
\label{Sec:ExtAlgLC}

We assume throughout that $ \mc{A}$ is unital and complete.
\begin{Def}
\label{Def:ToepAlg}
If $\mc{B}$ is a smooth GCP, we define $\ToepTop $, the Toeplitz algebra of $\mc{B}$, as the closed subalgebra of $\mc{B}\otimes_\pi \stoe$ (completed projective tensor product) generated by 
$$\mc{A}\otimes 1,\;S(\mc{E})\otimes S,\;\bar S(\mc{E})\otimes \bar S.$$
We denote $\ToepTopAlg$ the algebra generated by the same generators, but without closure.
\end{Def}
The canonical inclusion will be denoted by $\iota \colon \mc{A}\to \ToepTop $.

\begin{Rem}
In the $C^*$-setting, the analogous construction obtained by taking the $C^*$-tensor product of a Pimsner algebra with the $C^*$-nuclear Toeplitz algebra, yields exactly Pimsner's $\Tt_E$ -- as can be seen by applying the gauge invariance theorem (see for example \cite{Katsura04}, section 6).
\end{Rem}

\begin{Prop}
\label{Prop:ExtLisse}
Let $\mc{B}$ be a smooth GCP and $\ToepTop $ the associated Toeplitz algebra. There is a linearly split extension:
\begin{equation}
\label{Eqn:ExtLisse}
\xymatrix{ 0\ar[r] &\Ker \overline{\pi}\ar[r]& \ToepTop \ar[r]^-{\overline{\pi}}& \mc{B} \ar[r]&0.}
\end{equation}
\end{Prop}

\begin{proof}
We start from the linearly split exact sequence of locally convex algebras
$$
\xymatrix{ 0\ar[r] &\sco \ar[r]& \stoe \ar[r]^-{\pi}& LC(U) \ar[r]&0.}
$$
The definition of $\stoe $ entails that there is a linear splitting $\rho \colon LC(U) \to \stoe$. We can tensor this extension of locally convex algebras by $\mc{B}$. Since the initial extension is split, so is the resulting extension:
$$
\xymatrix{ 0\ar[r] & \mc{B} \TensTop  \sco \ar[r]& \mc{B} \TensTop \stoe \ar[r]& \mc{B}  \TensTop LC(U) \ar[r]&0.}
$$
Let $\overline{\pi}$ be the restriction of $\id \otimes \pi$ to $\ToepTop \subseteq \mc{B} \TensTop \stoe $. This map is continuous as a tensor product of two continuous maps and we get:
$$
\xymatrix{ 0\ar[r] &\Ker \overline{\pi}\ar[r]& \ToepTop \ar[r]^-{\overline{\pi}}& \im \overline{\pi} \ar[r]&0.}
$$
This sequence is linearly split because $\id \otimes \rho \colon \mc{B} \TensTop LC(U) \to  \mc{B} \TensTop \Toep$ restricts to a splitting $\overline{\rho}$ of $\overline{\pi}$ on $\im \overline{\pi}$.

To see this, denote by $\Delta$ the closed subalgebra of $\mc{B} \otimes LC(U)$ generated by $a \otimes 1$, $S_\xi \otimes U$ and $\bar S_\xi \otimes \bar U$ for $a \in \mc{A}$ and $\xi \in \Modl $. It is clear that $\overline{\pi}(\ToepTopAlg) \subseteq \Delta$. Since $\overline{\pi}$ is continuous and $\Delta$ is closed, we actually get $\overline{\pi}(\ToepTop) \subseteq \Delta$.

Using the left and right ``scalar products'' of Remark \ref{Rem:SPimsner}, any finite product of generators can be written either $ S_{\xi_1} \cdots S_{\xi_k} \otimes U^k $ if $k > 0$ or $\bar S_{\xi_1} \cdots \bar S_{\xi_{-k}} \otimes \bar U^{-k}$ if $k <0$ or $a \otimes 1$ if $k = 0$.

In any case, the splitting $\overline{\rho}$ send this product to an element of $\ToepTop  $ -- either $ S_{\xi_1} \cdots S_{\xi_k} \otimes S^k $ if $k > 0$ or $\bar S_{\xi_1} \cdots \bar S_{\xi_{-k}} \otimes \bar S^{-k}$ if $k <0$ or $a \otimes 1$ if $k = 0$. Since on the one hand, $\overline{\rho}$ is continuous and $\ToepTop $ is closed, and on the other hand the finite sums of finite products of generators are dense in $\Delta$ (by definition), it appears that $\overline{\rho}(\Delta ) \subseteq \ToepTop $. Consequently $\Delta = \im \overline{\pi}$, and in particular, $\im \overline{\pi}$ is closed.

Let us now prove that $\im \overline{\pi}$ is isomorphic to $\mc{B}$. The restriction of the evaluation map 
$$ev_1 \otimes \id \colon \mc{B} \TensTop LC(U) \longrightarrow  \mc{B} \TensTop \mC \simeq \mc{B},$$
where $ev_1(f) = f(1)$, clearly yields a continuous morphism $\im \overline{\pi} \to \mc{B}$.

\smallbreak

The Definition \ref{Def:SPimsner} of a smooth GCP entails that $t \mapsto \gamma_t(b)$ defines a continuous map $\mc{B} \to C^\infty(S^1 \to \mc{B})$. One can identify
$$C^\infty(S^1 \to \mc{B}) \simeq \mc{B} \TensTop LC(U)$$
(see \cite{CyclicHom}, p.62), hence we get a continuous map
$$\mc{B} \to \mc{B} \TensTop LC(U) .$$
It is readily checked that these maps are inverse to one another. The density of generators hence implies that $\im \overline{\pi}$ is isomorphic to $\mc{B}$ as locally convex algebras.
\end{proof}

\section{Morita equivalence of the kernel and base algebras}
\label{MoritaContext}

\begin{Prop}
\label{Prop:Kernel}
The kernel $\mc{C} := \Ker \overline{\pi}$ of \eqref{Eqn:ExtLisse} is the locally convex algebra generated by
$\mc{A}\otimes e$, $S(\mc{E})\otimes Se$ and $\bar S(\mc{E})\otimes e\bar S$ in $\mc{B} \TensTop \stoe$.
\end{Prop}

\begin{proof}
It is clear that the algebra $\mc{C}$ generated by $\mc{A}\otimes e$, $S(\mc{E})\otimes Se$ and $\bar S(\mc{E})\otimes e\bar S$ is contained the kernel. As the tensored sequence 
$$
\xymatrix{0\ar[r] & \mc{B} \otimes_\pi\sco\ar[r]& \mc{B} \otimes_\pi\stoe\ar[r]^-{\id_{\mc{B}} \otimes\pi}& \mc{B} \otimes_\pi LC(U)\ar[r]&0}
$$
stays exact, $\Ker(\overline{\pi})=\Ker(\id_{\mc{B}} \otimes \pi)\cap \ToepTop $.

Now let $x\in \Ker(\overline{\pi})$, $\ee>0$ and fix a continuous seminorm $p$ on $\ToepTop $. If $\overline{\rho}$ denotes again the splitting as in the proof of Proposition \ref{Prop:ExtLisse}, we can choose a continuous seminorm $q$ such that there is $c_{\overline{\rho}}$ with $p(\overline{\rho}(y))\leqslant c_{\overline{\rho}} \, q(y)$, and take a continuous seminorm $p'$ on $\ToepTop $ such that $q(\overline{\pi}(x'))\leqslant c_{\overline{\pi}} \, p'(x')$ for all $x'$. As $x\in \ToepTop $, we may choose a finite sum $x_0$ in $\ToepTopAlg$ with $p(x-x_0),p'(x-x_0)<\ee$. Then $x_0-\overline{\rho}(\overline{\pi}(x_0))\in\Ker(\overline{\pi})$ and
$$p(x-(x_0-\overline{\rho}(\overline{\pi}(x_0))))\leqslant p(x-x_0)+p(\overline{\rho}(\overline{\pi}(x-x_0)))\leqslant (1+c_{\overline{\rho}}c_{\overline{\pi}})\ee.$$ 
Now $x_0-\overline{\rho}(\overline{\pi}(x_0))$ is easily seen to be an element in the ideal generated by $1\otimes e$ in $\ToepTop $, and thus the other inclusion is proved.
\end{proof}

\begin{Rem}
A consequence of this Lemma is that $\mc{C}$ is included in the algebra $\mc{B} \TensTop \sco$. Every $x \in \mc{B} \TensTop \sco$ (and hence every $x \in \mc{C}$) has a representation of the form $\sum_{k,l} x_{k,l} \otimes S^k e \overline{S}^l$, where the $x_{k,l}$ are rapidly decreasing in $k,l$.
\end{Rem}

To go further in the analysis of the kernel, we rely on the algebra being tame. 
\begin{Not}
\label{Not:Xi}
Let $\xi_i$ be the sequence of $\mc{E}$-multivectors of the Definition \ref{Def:PiCrGTame}. We set
\begin{align*}
\Xi_i&:=S_{\xi_i}\otimes S^{|\xi_i|}e &
&\text{ and }
&
\bar\Xi_i&:= \bar S_{\xi_i}\otimes e{\bar S}^{|\xi_i|}.
\end{align*}
\end{Not}

\begin{Lem}
\label{into} 
For any $x \in \mc{C}$, the sequence $\left ( \bar\Xi_i x\Xi_j \right )_{i,j}$ is rapidly decreasing. Moreover, $\bar\Xi_i x\Xi_j \in \phi(\mc{B})$ for any $i,j$,  where $\phi \colon \mc{B} \to \mc{B} \TensTop \sco$ is given by $\phi(x) = x \otimes e$.
\end{Lem}

\begin{proof}
Let $x\in  \mc{C}$, and assume first that $x=\sum_{k,l} x_{k,l}\otimes S^ke \bar S^l$ is a finite sum. 
\begin{align*}
\bar\Xi_ix\Xi_j&=\sum_{k,l} \bar S_{\xi_i}x_{k,l} S_{\xi_j}\otimes e{\bar S}^{|\xi_i|}S^ke{\bar S}^{l} S^{|\xi_j|}e\\
&=\sum_{k,l} \bar S_{\xi_i}x_{k,l} S_{\xi_j}\otimes \delta_{|\xi_i|,k}\delta_{l,|\xi_j|}e\\
&=\bar S_{\xi_i}x_{|\xi_i|,|\xi_j|} S_{\xi_j}\otimes e,
\end{align*}
which is clearly an element of $\phi( \mc{B})$. This property extends by continuity to $\mc{C}$.

We now need to check that $\big( \bar S_{\xi_i} x_{|\xi_i|,|\xi_j|} S_{\xi_j} \big)_{i,j}$ is rapidly decreasing in $i$ and $j$. This can be seen using the fact that $x_{k,l}$ is rapidly decreasing and the polynomial growth of $\bar S_{\xi_i}$, $S_{\xi_j}$.

More precisely, take a polynomial $R(i,j)$ and fix a continous seminorm $p$ on $\mc{B}$. We prove that $p \left(  \bar S_{\xi_i} x_{|\xi_i|,|\xi_j|} S_{\xi_j} R(i,j) \right)$ is bounded above. Since we are considering a tame algebra, $p(\bar S_{\xi_i}) p(S_{\xi_j})$ is bounded by a polynomial $Q_0(i,j)$. In turn, we can give a bound for $|Q_0(i,j) R(i,j)|$ by using an ``increasing'' polynomial $Q(i,j)$. By ``increasing'', we mean that
$$  i \leqslant i' \text{ and } j \leqslant j' \implies Q(i,j) \leqslant Q(i',j') .$$
Using the relations $|X^p Y^q| \leqslant X^{2 p} + Y^{2 q}$, it follows that $|Q_0(i,j) R(i,j)|$ has upper bound $i^{d_1} + j^{d_2} + K$, for $d_1, d_2$ and $K$ large enough. Consider now
$$P(k,l) = \sum_{k',l' = 1}^{N} Q(N(k-1) +k', N(l - 1) + l').$$ 
This is clearly an increasing polynomial in the above sense.

Since $\big( p(x_{k,l}) \big)_{k,l}$ is rapidly decreasing in $k,l$, there is an upper bound $C$ such that $ p(\xi_{k,l}) P(k,l) \leqslant C $, for all $k,l \in \N$. Let us estimate
$$
 p \left( \bar S_{\xi_i} x_{|\xi_i|,|\xi_j|} S_{\xi_j} \right) \leqslant p(x_{|\xi_i|, |\xi_j|}) Q(i,j) \leqslant p(\xi_{k,l}) P(k_0,l_0) \leqslant p(\xi_{k,l}) P(k,l) \leqslant C
$$
where $k_0$ and $l_0$ are the integer parts of $i / N$ and $j/N$, respectively, $k = |\xi_i|$ and $l = |\xi_j|$. It follows from Remark \ref{Rem:CroissMod} that $k_0 \leqslant k$. By design, $Q(i,j) \leqslant P(k_0,l_0)$, which proves the second inequality. $P$ is increasing and this proves the third inequality.
\end{proof}

\begin{Lem}
\label{approx}
If $\mc{B}$ is a tame smooth GCP, then for any $x \in \mc{B} \TensTop \sco$,
\begin{enumerate}[(i)]
\item
for any $I \in \N$, $\sum_{i=1} ^I \Xi_i \bar \Xi_i x$ is an element of $\ToepTop $;
\item
$\sum_{i=1} ^I \Xi_i \bar \Xi_i x \xrightarrow{I \to \infty} x$ in $\ToepTop $.
\end{enumerate}
\end{Lem}

\begin{proof}
Let $x = \sum_{m,n} x_{m,n} \otimes S^m e \bar S^n$, we calculate:
 \begin{align*}
\left(\sum_{i=0}^I\Xi_i\bar\Xi_i \right)X=&\sum_{m,n=0}^\infty \sum_{i=0}^I S_{\xi_i}\bar S_{\xi_i}x_{m,n}\otimes S^{|\xi_i|}e\bar S^{|\xi_i|}S^me\bar S^n\\
=& \sum_{i=0}^I \sum_{n=0}^\infty S_{\xi_i}\bar S_{\xi_i} x_{|\xi_i|,n}\otimes S^{|\xi_i|} e\bar S^{n},
\end{align*}
which proves that $\sum_{i =1}^I \Xi_i \bar \Xi_i X $ is in $\ToepTop $, as long as the sum converges.

\smallbreak

We know that $x_{m,n}$ is rapidly decreasing in $m,n$. We need to show that the sequence $ S_{\xi_i}\bar S_{\xi_i} x_{|\xi_i|,n}\otimes S^{|\xi_i|} e\bar S^{n}$ is rapidly decreasing $i$ and $n$.

We use the same techniques as in Lemma \ref{into}: we fix a continuous seminorm $p$ and a polynomial $R(i,n)$. We can majorise the sequence
$$\left|R(i,n) p(\Xi_i) p(\bar \Xi_i) \right|$$
by a polynomial $Q(i,n)$, which we can choose increasing in the previous sense (see the proof of Lemma \ref{into}). 

Since $x_{m,n}$ is rapidly decreasing and  $S^m e \bar S^n$ has polynomial growth in $m,n$, the sequence $\left( x_{m,n} \otimes S^m e \bar S^n \right)_{m,n}$ is rapidly decreasing. Considering the polynomial $ P(m,n) = \sum_{m' =1}^N Q(N(m-1) + m') $, we can find a constant $C$ such that for any $m,n \in \N$, $ p \big( x_{m,n} \otimes S^m e \bar S^n \big)P(m,n) \leqslant C$. As a consequence, we get the inequalities
$$
p\big(\Xi_i \bar \Xi_i (x_{|\xi_i|,n} \otimes S^{|\xi_i|} e \bar S^n) \big) \leqslant p\big( x_{|\xi_i|,n} \otimes S^{|\xi_i|} e \bar S^n \big) P(m_0,n) \leqslant p\big( x_{m,n} \big) P(m,n) \leqslant C,$$
using the same arguments as in Lemma \ref{into}. Therefore, the sequence in $I \in \N$
$$ \left( \sum_{i=0}^I \sum_{n=0}^\infty S_{\xi_i}\bar S_{\xi_i} x_{|\xi_i|,n}\otimes S^{|\xi_i|} e\bar S^{n} \right)_I $$
is nothing but the partial sums of a summable series. Hence, it admits a finite limit and we will prove that this limit is $x$. 

For any degree $m \in \N$, the condition \eqref{Eqn:Repere} shows that there are two integers $l_1$ and $l_2$ such that $\sum_{i = l_1}^{l_2} S_{\xi_i} \bar S_{\xi_i} = 1_\mc{A}$. Algebraically, we see that for any fixed $m$, if $I \geqslant l_2$, then
$$ \left (\sum_{i = 0}^I \Xi_i \bar \Xi_i \right ) x_{m,n} \otimes S^{m} e \bar S^n = x_{m,n} \otimes S^m e \bar S^n .$$
Hence, for each fixed $m$, we get the identity. The linear forms $\mc{C} \to \mc{B}$:
$$ \sum_{m,n} x_{m,n} \otimes S^m e \bar S^n \longmapsto x_{m_0,n_0} ,$$
correspond to projections onto entries in the matrix. By evaluating the linear forms on the partial sums, it follows that the (unique) limit is $x$.

\bigbreak

Thus, the $\sum_{i= 0}^I \Xi_i \bar \Xi_i $ form an approximate unit for $\mc{C}$.
\end{proof}

\begin{Th}
\label{context}
If $\mc{B}$ is a tame smooth GCP, then there is a Morita bicontext between $\mc{C}$ and $\mc{A}$. If we denote by $\iota_\mc{C}$ the inclusion of $\mc{A}$ into $\mc{C}$ defined by $a \mapsto a \otimes e$, the bicontext is given as follows:
\begin{itemize}
\item
 from $\mc{C}$ to $\mc{A}$ by $( \id, \mc{C}, \iota_\mc{C}, \Xi_l, \bar \Xi_l)$;
\item 
 from $\mc{A}$ to $\mc{C}$ by $( \iota_\mc{C}, \mc{C}, \id, 1 \otimes e, 1 \otimes e)$.
\end{itemize}
In particular, if $H$ is a $\sco $-stable and diffotopy-invariant functor, the canonical inclusion $\iota_\mc{C} \colon  \mc{A} \inj \mc{C}$ induces an isomorphism $H(\mc{A}) \simeq H(\mc{C})$.
\end{Th}

\begin{proof}
Let us first check that we have a Morita context from $\mc{C}$ to $\mc{A}$, by verifying the conditions of Definition \ref{contextcond}: the morphisms clearly are isomorphisms onto subalgebras of  $\mc{C}$. It remains to check the three conditions of Definition \ref{contextcond}. Let $x = \sum_{k,l} x_{k,l} \otimes S^k e \bar S^l$, where $x_{k,l}$ is a homogeneous element of degree $k - l$.
\begin{enumerate}[(i)]
\item $\bar \Xi_i X \Xi_j$ is in $\iota_\mc{C}(\mc{A})$: we saw in Lemma \ref{approx} that
$$
\bar \Xi_i \left( \sum_{k,l} x_{k,l} \otimes S^k e \bar S^l \right) \Xi_j = \bar S_{\xi_i} x_{|\xi_i|,|\xi_j|} S_{\xi_j} \otimes e .
$$
Counting the degree proves that $\bar S_{\xi_i} x_{|\xi_i|,|\xi_j|} S_{\xi_j}$ has degree $0$ and therefore, it is an element of $\mc{A}$.  Thus, $\bar \Xi_i X \Xi_j \in \iota_\mc{C}(\mc{A})$.
\item 
This is a consequence of Lemma \ref{into}.
\item 
This is a consequence of Lemma \ref{approx}.
\end{enumerate}

It is readily checked that $1 \otimes e$ induces a Morita context:
\begin{enumerate}[(i)]
\item
$1 \otimes e$ is in $\mc{C}$, hence $(1 \otimes e )x(1 \otimes e) \in \mc{C}$, for any $x \in \mc{C}$;
\item 
the sequence is finite and therefore rapidly decreasing;
\item 
 $(1 \otimes e) (1 \otimes e) (a \otimes e) = a \otimes e$, thus the partial sums converge in norm.
\end{enumerate}
It is now enough to check that the two Morita contexts are left and right compatible, by checking Definition \ref{Def:BicontextMorita}:
\begin{enumerate}[(i)]
\item
This condition is trivially verified, because the underlying algebra is $\mc{C}$ itself.
\item 
$ (1 \otimes e) \left(S_{\xi_l} \otimes S^{|\xi_l|} e \right)$ vanishes if $|\xi_l| \neq 0$. Moreover, the only element with degree $0$ is $\xi_1 = 1$ and then $S_{\xi_1} \otimes S^{|\xi_1|} e = 1 \otimes e$, thus $\iota_{\mc{C}}(a) (1 \otimes e) = \iota_\mc{C}(a)$. In the same way, $\left( \bar S_{\xi_l} \otimes  e \bar S^{|\xi_l|} \right) (1 \otimes e) $ equals either $0$ (when $|\xi_l| \neq 1$) or $1 \otimes e$ (when $|\xi_l| = 1$) and in this case $ (1 \otimes e) \iota_\mc{C}(a) = \iota_\mc{C}(a)$.
\end{enumerate}
\end{proof}

\section{Equivalence of the Toeplitz and base algebras}
\label{Sec:ToepBase}

In this section, we will prove that the smooth Toeplitz algebra $\ToepTop $ is ``$H$-equivalent'' to the base algebra $\mc{A}$, for any half-exact, $\sco$-stable and diffotopy-invariant functor $H$.

For that purpose, we consider a quasihomomorphism $(\alpha, \overline{\alpha}) \colon \ToepTop \rrarrow \mc{B} \otimes \Toep \trianglerighteq \mc{C} $. We will prove that the morphism of abelian groups $H(\alpha, \overline{\alpha})$ is both left- and right-invertible -- therefore showing that it is invertible.

We use Theorem \ref{context} to prove that it is right-invertible. It will be more difficult to find a left inverse. We will need to enlarge the image of the quasihomomorphism. More explicitly, we will have to introduce an algebra $\bar{\mc{C}}$ and a quasihomomorphism $(\Phi,\Psi) \colon \ToepTop\rrarrow Z\big( \mc{B} \TensTop\Toep\TensTop\Toep \big)\trianglerighteq \bar{\mc{C}}$ (Lemma \ref{Lem:qMorph2}), then exhibit a Morita context  (Proposition \ref{Prop:ContextMorita}) before we can conclude (Theorem \ref{PMV}).

Recall that $\mc{C}$ is the subalgebra of $\mc{B} \TensTop \stoe$ generated by
$\mc{A}\otimes e$, $S(\mc{E})\otimes Se$ and $\bar S(\mc{E})\otimes e\bar S$.
\begin{Lem}
\label{Lem:qMorphisme}
Let $\mc{B}$ be a smooth GCP, the morphisms $\alpha, \bar{\alpha} \colon \ToepTop \to \mc{B} \TensTop \Toep $ given by
\begin{align*}
\alpha &= \id_{\ToepTop}
&
\bar{\alpha}&= \Ad(1 \otimes S),
\end{align*}
where $\Ad(1 \otimes S)(x) = (1 \otimes S)x (1 \otimes \bar S)$, define a quasihomomorphism $(\alpha,\bar\alpha)\colon  \ToepTop\rrarrow \mc{B} \otimes\Toep\trianglerighteq \mc{C}$.
\end{Lem}

\begin{proof}
It suffices to see that the ideal $\mc{C}$ contains $(\alpha - \overline{\alpha})(\ToepTop )$. Since 
\begin{align*}
\big(\alpha - \overline{\alpha} \big)(a \otimes 1) &= a \otimes 1 - (1 \otimes S)(a \otimes 1)(1 \otimes \bar S) = a \otimes (1 - S \bar S) = a \otimes e ,\\
\big(\alpha - \overline{\alpha} \big)(S_{\xi} \otimes S)&= S_\xi \otimes  S - (1 \otimes S)(S_\xi \otimes  S)(1 \otimes \bar S) = S_\xi  \otimes S e ,\\
\big(\alpha - \overline{\alpha} \big)(\bar S_{\xi} \otimes \bar S)&= \bar S_{\xi} \otimes \bar S - (1 \otimes S)(\bar S_\xi \otimes \bar S)(1 \otimes \bar S) = \bar S_\xi \otimes e\bar S,
\end{align*}
the images of the generators are all in $\mc{C}$, and this suffices by Proposition \ref{quasiprops}.
\end{proof}

\begin{Def}
\label{Def:barCc}
If $\mc{B}$ is a smooth GCP, we denote by $\bar{\mc{C}}$ the closed subalgebra of $\mc{B} \TensTop \sco \TensTop \Toep $ generated by
$$ S_{\xi_I} \bar S_{\eta_J} \otimes S^p e \bar S^q \otimes S^{p'} \bar S^{q'}, $$
where $p,q \in \N$, $\xi_I \in \Modl ^{(p+p')}$ and $\eta_J \in \Modl ^{(q + q')}$.
\end{Def}

We define:
\begin{itemize}
\item
 $\Phi \colon \ToepTop\to Z\big( \mc{B} \TensTop  \Toep \TensTop \Toep \big)$ as the restriction of $\id_{ \mc{B}} \otimes (\phi_t)_t$, where $(\phi_t)_t \colon \Toep\to Z(\Toep\TensTop\Toep)$ is the diffotopy of Lemma \ref{ahomotopy},
\item 
and $\Psi \colon \ToepTop\to Z\big( \mc{B} \TensTop\Toep\TensTop\Toep\big)$ as the restriction of the map $\Ad(1\otimes \hat S)$, independant of $t$, given by $\Ad(1\otimes\hat S)(x)=(1\otimes \hat S)(x \otimes 1)(1\otimes\hat{\bar S})$ for all $x \in \mc{B} \TensTop\Toep$.
\end{itemize}
 
\begin{Lem}
\label{Lem:qMorph2}
Let $\mc{B}$ be a smooth GCP. The morphisms $\Phi, \Psi$ above define a quasihomomorphism $(\Phi,\Psi) \colon \ToepTop\rrarrow Z\big( \mc{B} \TensTop\Toep\TensTop\Toep \big)\trianglerighteq Z\bar{\mc{C}}$.
\end{Lem}

\begin{proof}
We have to check that $(\Phi - \Psi)(\ToepTop ) \subseteq Z \bar{\mc{C}}$ and that this algebra is $\Psi(\ToepTop)$-stable. We can check this on the generators for all $t$ by Proposition \ref{quasiprops}:
\begin{multline*}
(\Phi_t - \Psi_t)(a \otimes 1) = a \otimes \hat 1 - a \otimes \hat S \hat{\bar S} = a \otimes \hat e,  \\
\shoveleft{(\Phi_t - \Psi_t)(S_\xi \otimes S) = S_\xi \otimes \left( \hat S (\hat 1 - \hat e) + f(t) \check S + g(t) \hat S \hat e \right) - S_\xi \otimes \hat S \hat S \hat{\bar{S}} } \\
\shoveright{=  S_\xi \otimes \left( f(t) \check S + g(t) \hat S \hat e \right),} \\
\shoveleft{(\Phi_t - \Psi_t)(\bar S_\xi \otimes \bar S) = \bar S_\xi \otimes \left( (\hat 1 - \hat e) \hat{\bar{S}} + f(t) \check{\bar S} + g(t) \hat{ \bar e} \hat{\bar S} \right) - \bar S_\xi \otimes  \hat S \hat{\bar{S}} \hat{\bar{S}}}\\
=\bar S_\xi \otimes \left( f(t) \check{\bar S} + g(t) \hat{ e} \hat{\bar S}  \right).
\end{multline*}
All these elements are clearly in $\bar{\mc{C}}$. $\Psi(\ToepTop )$-stability follows from the algebraic relations, by an application of Lemma \ref{triviallem}.
\end{proof}

\begin{Prop}
\label{Prop:ContextMorita}
Let $\mc{B}$ be a tame smooth GCP, there is a Morita context $\left(\id_{\bar{\mc{C}}'}, \bar{\mc{C}}, \iota_1,\Xi_i',\bar \Xi_j' \right)$ from $\bar{\mc{C}}$ to $\ToepTop $, given by:
\begin{itemize}
\item
 the sequences $\Xi_i' = \Xi_i \otimes 1 $ and $\bar \Xi_i' = \bar \Xi_i \otimes 1$, using Notation \ref{Not:Xi};
\item 
 the morphism $\iota_1 \colon \ToepTop \to \bar{\mc{C}}$ defined by
\begin{align*}
\iota_1(a \otimes 1) &= a \otimes e \otimes 1 
&
\iota_1(S_\xi \otimes S) &= S_\xi \otimes e \otimes S 
&
\iota_1(\bar S_\xi \otimes \bar S) &= \bar S_\xi \otimes e \otimes \bar S ,
\end{align*}
for $a \in \mc{A}$ and $\xi \in \Modl $.
\end{itemize}
\end{Prop}

Notice that the Morita context in fact induces a morphism from $\bar{\mc{C}}$ to $\iota_1( \ToepTop ) \otimes \sco$. We identify $\iota_1(\ToepTop ) \simeq \ToepTop $ to obtain a map $\theta \colon \bar{\mc{C}} \to \ToepTop \TensTop \sco $.

\begin{proof}
We check the conditions of Definition \ref{contextcond}.
\begin{enumerate}[(i)]
\item
Algebraically, if we take $Y =S_{\xi_I} \bar S_{\xi_J} \otimes S^p e S^q \otimes S^{p'} S^{q'} $ as in Definition \ref{Def:barCc}, 
\begin{multline*}
\bar \Xi_i' Y \Xi_j' =\\
= \left( \bar S_{\xi_i} \otimes e \bar S^{|\xi_i|} \otimes 1\right) \left( S_{\xi_I} \bar S_{\eta_J} \otimes S^p e S^q \otimes S^{p'} S^{q'} \right) \left( S_{\xi_j} \otimes S^{|\xi_j|} e \otimes 1 \right) = \\
= \delta_{|\xi_i|,p} \delta_{|\xi_j|,q} \bar S_{\xi_i} S_{\xi_I} \bar S_{\eta_J} S_{\xi_j} \otimes e \otimes S^{p'} \bar S^{q'}.
\end{multline*}
A simple computation of degree shows that $\bar \Xi_i' Y \Xi_i' $ is in $\iota_1(\ToepTop )$.
\item 
Since $\bar{\mc{C}}$ is a subalgebra of $\mc{B} \TensTop \sco \TensTop \Toep $, by using the canonical isomorphism with $\big( \mc{B} \TensTop \Toep \big) \TensTop \sco $ all the elements of $\bar{\mc{C}}$ can be written as
\begin{equation}
\label{Eqn:FormeGbarCc}
\sum_{k,l} Y_{k,l} \otimes S^k e \bar S^l ,
\end{equation}
where $(Y_{k,l})_{k,l}$ is a rapidly decreasing sequence of $\mc{B} \TensTop \Toep$. In the same way as in Lemma \ref{into}, the sequences $\bar \Xi'_i$ and $\Xi'_j$ have polynomial growth in $i,j$. The sequence $S^k e \bar S^l$ also have polynomial growth. The rapidly decreasing sequence $Y_{k,l}$ can therefore ``absorb'' these terms and $\bar \Xi_i' Y \Xi_j'$ is a rapidly decreasing sequence of $\mc{B} \TensTop \Toep $.
\item 
Algebraically
\begin{multline*}
\Xi_i' \bar \Xi_i' Y  = \left( S_{\xi_i} \bar S_{\xi_i} \otimes S^{|\xi_i|} e \bar S^{|\xi_i|} \otimes 1\right) \left( S_{\xi_I} \bar S_{\eta_J} \otimes S^p e S^q \otimes S^{p'} S^{q'} \right) = \\
= \delta_{|\xi_i|,p}  S_{\xi_i} \bar S_{\xi_i}  S_{\xi_I} \bar S_{\eta_J} \otimes S^p e S^q \otimes S^{p'} \bar S^{q'}
\end{multline*}
and equation \eqref{Eqn:Repere} proves that we only have to sum the $\xi_i$ such that $|\xi_i| = p$ to ``reconstruct'' $Y$. Just like before, we can start from elements of $\bar{\mc{C}}$ of the form \eqref{Eqn:FormeGbarCc} with a rapidly decreasing $Y_{k,l}$. The end of the argument is  very similar to that of Lemma \ref{approx}: the sequence $\Xi_i' \bar \Xi_i' Y_{|\xi_i|,l}$ is rapidly decreasing and therefore summable, and then we prove that the limit can only be $\sum_{k,l} Y_{k,l} \otimes S^k e \bar S^l $, by projecting onto the coefficients of this matrix.
\end{enumerate}
\end{proof}

We define a morphism $\sigma \colon \ToepTop  \to \ToepTop \TensTop \sco$ for $x \in \ToepTop $ by
$$ \sigma(x) = x \otimes e ,$$
where we identify $\sco$ with the algebra generated by $S^p e \bar S^q$ for $p,q \in \N$, just as in Definition \ref{Def:ToepLisse}.

\begin{Lem}
\label{Lem:InvPsi1Phi1}
If $\mc{B}$ is a tame smooth GCP and if $H$ is a $\sco$-stable and split exact functor, 
$$ H(\theta) \circ H(\Phi_1, \Psi_1) = H(\sigma) ,$$ 
up to identification of $\iota_1(\ToepTop )$ with $\ToepTop $. Consequently, $ H(\Phi_1, \Psi_1)$ is left-invertible.
\end{Lem}

\begin{proof}
Since $\phi_t$ is unital,
\begin{align*}
\Phi_1(a\otimes 1)&=a\otimes 1\otimes 1=(\Ad(1\otimes \hat S)(a\otimes \hat 1))+(a\otimes \check 1),\\
\Phi_1(S_\xi\otimes S)&=S_\xi\otimes(\hat S^2\hat{ \bar S}+\check S)= \left(\Ad(1\otimes \hat S)(S_\xi\otimes \hat S) \right)+(S_\xi\otimes \check S).
\end{align*}
If we use the notation $\iota_1$ introduced in Proposition \ref{Prop:ContextMorita}, we see that $\Image(\Ad(1\otimes \hat S))$ and $\Image(\iota_1)$ are orthogonal because $\bar Se=0=eS$. We can therefore apply Proposition \ref{quasiprops} and deduce
\begin{align*} 
H(\Phi_1,\Psi_1)= H\big(\Ad(1\otimes \hat S)\oplus \iota_1,\Ad(1\otimes \hat S)\big)= H(\iota_1).
\end{align*}
Using Theorem \ref{context}, it is easy to show that $\theta \circ \iota_1$  is equal to $\sigma$, up to identification of $\iota_1(\ToepTop )$ with $\ToepTop $. Since $H$ is $\sco$-stable, $H(\sigma)$ is an isomorphism, which completes the proof.
\end{proof}

Let us recall that $\iota \colon \mc{A} \to \ToepTop $ is the canonical inclusion, while the injection $\iota_\mc{C} \colon \mc{A} \to \mc{C}$ was defined in Theorem \ref{context} and the quasihomomorphism $(\alpha,\bar\alpha) \colon  \ToepTop\rrarrow \mc{B} \otimes\Toep\trianglerighteq \mc{C}$ in Lemma \ref{Lem:qMorphisme}.

\begin{Th}
\label{PMV} 
Let $H$ be a split-exact, diffotopy-invariant and $\sco$-stable functor from the category of locally convex algebras to that of abelian groups and $\mc{B}$ a tame smooth GCP. The inclusion $\iota \colon  \mc{A} \inj \ToepTop $ induces an isomorphism
 $$H(\iota) \colon  H(\mc{A}) \xrightarrow{\sim} H(\ToepTop).$$
\end{Th}

\begin{proof}
Let $y:=H(\alpha,\bar\alpha)$. It suffices to show that $y$ is invertible. Indeed, as
$$y\circ H(\iota)=H(\alpha\circ\iota,\bar\alpha\circ\iota)=H(\iota_C)$$
and $H(\iota_C)$ is invertible by Theorem \ref{context}, so is $H(\iota)$.

It further suffices to show that $y$ is right and left invertible. To see right-invertibility, we use Proposition \ref{quasiprops} to see that
$$ H(\alpha,\bar\alpha)\circ H(\iota)=H(\alpha\circ\iota,\bar\alpha\circ\iota)=H(\iota_C) $$
and thus $H(\iota_C)$ is invertible by Theorem \ref{context}.

To see left-invertibility,  we show that there exists a morphism $z$ such that  $z \circ y$ is left-invertible. For $z$, we may take the morphism induced by the restriction  of $\Phi_0$ to a morphism $\mc{C}\to \bar{\mc{C}}$. In other words, we simply view $(\alpha,\bar\alpha)$ as a quasihomomorphism into a larger algebra, namely as the quasihomomorphism $(\Phi_0,\Psi_0)$.  As by Lemma \ref{Lem:InvPsi1Phi1}, $H(\Phi_1,\Psi_1)$ is invertible, and $H$ is diffotopy-invariant, we get the result. 
\end{proof}

\section{Six-term exact sequence and Chern character}
\label{Sec:6terms}

We now use the theory $kk$ to show that there are exact sequences in $kk$ and bivariant cyclic theory that are compatible with the Chern-Character and the boundary maps. In particular, we get, specializing the first algebra to $\mC$, such sequences in $K$-theory and  cyclic theory.

However, to give our first result, we do not really need $kk$-theory, but we obtain a six-term exact sequence for every half-exact, $\sco$-stable and diffotopy-invariant functor.
\begin{Def} A functor $H$ on the category of locally convex algebras with values in abelian groups is called
 \emph{half-exact} if for every linearly split short exact sequence
\[\xymatrix{ 0\ar[r]&\mc{A}\ar[r]&\mc{B}\ar[r]&\mc{C}\ar[r]&0}\] 
the sequence
\[\xymatrix{ H(\mc{A})\ar[r]&H(\mc{B})\ar[r]&H(\mc{C})}\]
obtained by applying $H$ is exact.
\end{Def}

It is easy to see that a \emph{half-exact} functor is always \emph{split-exact}. These functors are interesting because they yield six-term exact sequences, as we will see shortly. We remind the reader that $\mathscr{S} \mc{A}$ denotes the smooth suspension of the algebra $\mc{A}$.
\begin{Lem}
\label{Lem:HHexagone}
If $\mc{A}$, $\mc{B}$ and $\mc{C}$ are locally convex algebras and if $H$ is a half-exact, diffotopy-invariant and $\sco$-stable functor, then the extension
$$ 0 \longrightarrow \mc{A} \longrightarrow \mc{B} \longrightarrow \mc{C} \longrightarrow 0 $$
induces a six-term exact sequence:
\[\xymatrix{ H(\mc{A})\ar[r]&H(\mc{B})\ar[r]&H( \mc{C})\ar[d]\\
H(\mathscr{S} \mc{C})\ar[u]&H(\mathscr{S} \mc{B})\ar[l]&H(\mathscr{S} \mc{A}).\ar[l]
}\]
\end{Lem}

\begin{proof}
This is a well known result (see for instance \cite{CyclicHom}, p.48). One extends the half-exact functor into a homology theory $H_n$ by setting $H_n(\mc{A}) = H(\mathscr{S}^n A)$. Observe that $H_n$ is automatically Bott-periodic, either by carrying over the proof from \cite{{MR2240217}} or by universality of $kk$; hence the result follows.
\end{proof}

\begin{Th}
\label{generalth} 
If $\mc{B}$ is a tame smooth GCP, then for every half-exact, diffotopy-invariant and $\sco$-stable functor $H$ there is a six-term exact sequence
\[\xymatrix{ H(\mc{A})\ar[r]&H(\mc{A})\ar[r]&H(D)\ar[d]\\
H(\mathscr{S}D)\ar[u]&H(\mathscr{S}\mc{A})\ar[l]&H(\mathscr{S}\mc{A}).\ar[l]
}\]
\end{Th}

\begin{proof}
We just have to combine Lemma \ref{Lem:HHexagone} with extension \eqref{Eqn:ExtLisse} and Theorems \ref{context} and \ref{PMV} to get the six-term exact sequence.
\end{proof}

We will now use the version of $kk$ for $m$-algebras introduced in \cite{MR1456322}. We also refer to this article for the definition of the bivariant Chern-Connes character $kk_* \to HP_*$ (which essentially exists by universality of $kk$).

\begin{Th}
Let $\mc{B}$ be a smooth GCP. Suppose that $\mc{B}$ is an $m$-algebra, set $\mc{A}:= \mc{B}^{(0)}$, and assume that the Chern-Connes character yields an isomorphism
$$ kk_*(\mc{D}, \mc{A}) \otimes \mC \to HP_*(\mc{D},\mc{A}) \otimes \mC ,$$ 
then 
$$ kk_*(\mc{D}, \mc{B}) \otimes \mC \to HP_*(\mc{D},\mc{B}) \otimes \mC $$ 
is an isomorphism.

In particular, this holds for any $\mc{B}$ with commutative base $\mc{A}$.
\end{Th}

\begin{proof}
This follows by an application of the five Lemma to 
\[\scalebox{0.9}{\xymatrix{ &HP_*(\mc{D},\mc{A})\ar[rr]&&HP_*(\mc{D},\mc{A})\ar[dl]\\
  &&HP_*(\mc{D},\mc{B})\ar[ul]\\
kk_*(\mc{D},\mc{A})\ar[rr]\ar[uur]&{}&kk_*(\mc{D},\mc{A})\ar[dl]\ar[uur]\\
&kk_*(\mc{D},\mc{B})\ar[ul]\ar[uur]
}}\]
where the maps between $kk$ and cyclic theory denote the bivariant Chern-Connes character from \cite{MR1456322}, the others come from the long exact sequence.
\end{proof}

\section{Application to smooth crossed products}
\label{Sec:AppliPiCrG}

Let us start by noting that a Lie group action on a Pimsner algebra induces a smooth GCP:
\begin{Rem}
\label{Rem:ExGCP}
If $\alpha \colon G \action B$ is a pointwise continuous Lie group action on a $C^*$-Pimsner algebra $B$ which preserves the degree of homogeneous elements, we can extend it to an action $\widetilde{\alpha}$ of $\widetilde{G} = G \times S^1$ by adding the gauge action to it. In this case, the algebra $\mc{B}$ of $\widetilde{G}$-smooth elements of $B$ is a natural smooth subalgebra of $B$. The Lie algebra $\widetilde{\mathscr{G}}$ associated to $\widetilde{G}$ acts by derivations on $\mc{B}$. 

Moreover, $\mc{B}$ is a holomorphically closed Fréchet subalgebra $\mc{B}$ of $B$, equipped with the seminorms induced by the action of the universal enveloping algebra $U(\widetilde{G})$ of $\widetilde{G}$ (see for instance \cite{EltNCG}, Proposition 3.45). All $b \in \mc{B}$ are gauge smooth by construction, and it is then easy to check that the map $\mc{B} \to C^\infty(S^1 \to \mc{B})$ is continuous. Hence $\mc{B}$ is a smooth GCP.
\end{Rem}

To link  our construction with previous work  by Nest (\cite{CohomCyclZ}), we prove the following:
\begin{Prop}
If $\Aa \hat{\times}_\alpha \Z$ is a smooth crossed product in the sense of Definition 1.2 of \cite{CohomCyclZ}, then $\Aa\hat{\times}_\alpha\Z$ is a smooth GCP in the sense of Definition \ref{Def:SPimsner}.

Moreover, if $\alpha$ and $\alpha^{-1}$ have operator norms with respect to all defining seminorms of the smooth crossed product, then $\Aa \hat{\times}_\alpha \Z$ is tame.
\end{Prop}

\begin{proof}
Set  $\underline{\Aa} = \Aa \hat{\times}_\alpha \Z$. We start by proving that $\underline{\Aa}$ is a smooth GCP, using the notations of \cite{CohomCyclZ}. We define a gauge action by
$$ \gamma_t \left ( \sum_{n} a_n U^n \right ) = \sum_{n} e^{i n t} a_n U^n .$$
It is clear that $\underline{\Aa}^{(m)} = \Aa U^m$. Hence, $\underline{\Aa}^{(0)} = \Aa$ and $\underline{\Aa}^{(1)} = \Aa U$ and $\underline{\Aa}^{(0)}$ and $U$ clearly generate $\underline{\Aa}$ as a locally algebra.

All elements in $\underline{\Aa}$ are gauge smooth: indeed an infinitesimal generator $\partial $ a gauge action is given by:
$$ \partial \left ( \sum_{n} a_n U^n \right ) = \sum_{n} ( i 2 \pi n) a_n U^n .$$
By \cite{CohomCyclZ}, seminorms defining the topology on $\underline{\Aa}$ are:
$$ \left \| \sum_{n} a_n U^n \right \|_k = \sup_n \rho_k(n) \| a_n \|_k ,$$
where the $\rho_k(n)$ are defined in section 1 of \cite{CohomCyclZ}. The item (1) of Lemma 1.1 of \cite{CohomCyclZ} proves that $|n| \rho_k(n) \leqslant \rho_{k+1}(n)$. As a consequence,
\begin{multline}
\label{Eqn:Maj}
 \left \| \partial \left ( \sum_{n} a_n U^n \right ) \right \|_k = |i 2 \pi | \sup_n \rho_k(n) |n| \| a_n \|_k  \leqslant 2\pi \sup_n \rho_{k+1}(n) \| a_n \|_{k+1} \\
\leqslant  2 \pi  \left \| \partial \left ( \sum_{n} a_n U^n \right ) \right \|_{k+1},
\end{multline}
because the sequence of seminorms $\| \cdot \|_k$ is increasing. Hence, for any $ \sum_{n} a_n U^n$ in $\underline{\Aa}$, $\partial \left (\sum_{n} a_n U^n \right )$ is also in $\underline{\Aa}$ and $\sum_{n} a_n U^n$ is gauge smooth.

It is obvious from the definition of the seminorms on $\underline{\Aa}$ that the gauge action is isometric. Hence, the topology of $C^\infty(S^1 \to \underline{\Aa})$ is generated by the seminorms 
$$ \sum_{n} a_n U^n \mapsto  \left \| \partial^l \left ( \sum_{n} a_n U^n \right ) \right\|_k ,$$
for $k = 1, 2, \ldots $ and $l = 0, 1, \ldots $. By iterating the inequality \eqref{Eqn:Maj}, it is therefore easy to prove that the map $\underline{\Aa} \to C^\infty(S^1 \to \underline{\Aa})$ induced by the gauge action is continuous.

\bigskip

To see that $\underline{\Aa}^{(m)}$ is tame if $\alpha$ is isometric, note that it is clear that for any $m$, $ U^m$ provides a frame for $\underline{\Aa}^{(m)}$. Hence the sequence $(\xi_i)_i = (U^i)_i$ satisfies condition \eqref{Pt:PiCrGTame1} of Definition \ref{Def:PiCrGTame}. Using the definition of the seminorm $\| \, \|_k$, we get:
$$ \| U^m \|_k = \rho_k(n) = \sup_{i \leqslant k} \left ( \sum_{j = -n} ^n \| \alpha^j \|_i \right )^k = (2 n + 1 )^k ,$$
using the assumptions on $\alpha$ and $\alpha^{-1}$.
\end{proof}

\section{Application to quantum Heisenberg manifolds}
\label{Sec:AppliQHM}

\begin{Def}
Given two real numbers $\mu, \nu$ and an integer $c>0$, we define the \emph{Quantum Heisenberg Manifolds} (QHMs or $\QHM[c][\mu,\nu]$) as the $C^*$-completion of the algebra
$$ D_0 = \left\{ F \in C_c(\Z \to C_b(\R \times S^1)) \middle| F(p,x + 1, y) = e(-c p(y - p \nu)) F(p,x,y) \right\} $$
where $e(x) = e^{2 \pi i x}$, equipped with the multiplication:
\begin{equation}
\label{Eqn:Comp}
 (F_1 \cdot F_2)(p,x,y) = \sum_{q \in \Z} F_1(q,x,y) F_2(p-q,x - q 2 \mu,y- q 2 \nu).
\end{equation}
\end{Def}

In the following, we will write $\QHM $ instead of $\QHM[c][\mu,\nu]$ whenever there is no risk of confusion. One can prove that the involution is:
$$F^*(p,x,y) = \overline{F}(-p,x - 2 p \mu, y - 2 p \nu) .$$
We recall that
\begin{equation}
\label{Eqn:Period}
 \QHM ^{(n)} = \big\{ F(p,x,y) = \delta_{n,p} f(x,y) \big| f(x + 1, y) = e(-c n (y - n \nu)) f(x,y) \big\}.
\end{equation}

\begin{Lem}
\label{Lem:MorEqQHM}
Given $(c, \mu, \nu) \in \N^* \times \R^2$ and an integer $n \in \Z$, there is a frame of two elements $\xi_i^n$, $i = 1,2$ in $\QHM ^{(n)}$. In fact
\begin{align*}
(\xi_1^n)^* \xi_1^n + (\xi_2^n)^* \xi_2^n &= 1
&
\xi_1^n (\xi_1^n)^* + \xi_2^n (\xi_2^n)^* &= 1.
\end{align*}
Furthermore, we can choose $\xi_1^n$ and $\xi_2^n$ to be smooth functions.
\end{Lem}

\begin{proof}
On $D^{(0)}$, we get a frame from $\xi_1:=1$, $\xi_2:=0$. For $n\neq 0$, let $U$, $V$ be small neighbourhoods of $[0,\nicefrac{1}{2}]$ and $[\nicefrac{1}{2},1]$, respectively, and $f_1$, $f_2$ a partition of unity subordinate to $U$, $V$. Set $\chi_i = \frac{f_i}{\sqrt{f_1^2 + f_2^2}}$, then $\chi_1^2 + \chi_2^2 = 1$.

We define $\xi_1^n$ on $U\times S^1$ by setting $\xi_1^n(x,y) = \chi_1(x)$. Choosing $U$ and $V$ small enough $\xi_1$ may be assumed to vanish with all derivatives on the boundary of a fundamental domain, and using the equation \eqref{Eqn:Period}, $\xi_1^n$ can then be extended to an element of $\QHM ^{(n)}$ . 

A similar process can be applied to another fundamental domain and $\chi_2$ to obtain $\xi_2^n$. An easy computation on tailored fundamental domains yields
\begin{align*}
(\xi_i^n)^* \xi_i^n &= \overline{\chi_i} \chi_i
& 
 \xi_i^n (\xi_i^n)^* &= \chi_i \overline{\chi_i},
\end{align*}
and then $\chi_1^2 + \chi_2^2 = 1$ ensures that
\begin{align*}
(\xi_1^n)^*  \xi_1^n + (\xi_2^n)^*  \xi_2^n &= 1 
&
\xi_1^n (\xi_1^n)^* + \xi_2^n (\xi_2^n)^* &= 1.
\end{align*}
Finally, notice that if $\chi_i$ is smooth, then so is $\xi_i^n$.
\end{proof}

Recall that the \emph{Heisenberg group} $H_1$ is the subgroup of $GL_3(\R)$ of the matrices
\begin{align*} M(r,s,t):=
&\begin{pmatrix}
1 & s & t \\
0 & 1 & r \\
0 & 0 & 1
\end{pmatrix}
&
&\text{ where } r,s,t \in \R
\end{align*}
It has been proved (see \cite{RieffelDefQuant}, section 5) that the Heisenberg group acts on $\QHM $. We use the following expression for the action:
\begin{equation}
\label{Eqn:AlphQHM}
 \alpha_{(r,s,t)}(F)(x,y,p) = e\Big(-p\big(t + c s (x - r)\big)\Big) F(x-r,y-s,p).
\end{equation}
The infinitesimal generators $\partial _i$, for $i =1,2,3$ are
\begin{align*}
	\partial_1(F)(p,x,y) =& -\derp[F]{x}(p,x,y) & \partial_3(F)(x,y,p) =& - i 2 \pi p F(p,x,y) \\
\end{align*}
$$ \partial_2(F)(p,x,y) = - \derp[F]{y}(p,x,y) - i 2 \pi c p x F(p,x,y).$$

Note that $\alpha_{(0,0,t)}$ is just the gauge action. As $M(0,0,t)$ is in the center, the action of $H_1$ preserves the grading and $\partial_1$, $\partial_2$ commute with $\partial_3$. A short calculation shows:
\begin{align}
\label{Eqn:RelComm}
	[\partial_1, \partial_2] &= -c \partial_3 & [\partial_1, \partial_3] &= 0 & [\partial_2, \partial_3] &= 0.
\end{align}

\begin{Def} 
We define the \emph{smooth quantum Heisenberg manifold} $\mc{D}$ as the subalgebra of $H_1$-smooth elements of $D$.
\end{Def}
Note that $\mc{D}$ inherits the $\mZ$-grading from $D$, and that $\mc{D}^{(1)}$ is actually a bimodule over $\mc{D}^{(0)}$.

\begin{Prop}
The smooth quantum Heisenberg manifold is a Fréchet algebra and a tame smooth GCP.
\end{Prop}

\begin{proof}
We can apply Remark \ref{Rem:ExGCP} to see that QHM are Fréchet algebras and smooth GCP. It remains to show that $\mc{D}$ is tame. By Lemma \ref{Lem:MorEqQHM}, there is a frame with two elements for every $\mc{D}^n$. Denote $\xi_n$ the  sequence of frames ordered as in Notation \ref{Not:Xi}. Then the point \eqref{Pt:PiCrGTame1} of Definition \ref{Def:PiCrGTame} is clearly satisfied. It only remains to prove the polynomial growth.

Because scaling and adding do not change the equivalence-class of a family of seminorms, we may apply  Poincaré-Birkoff-Witt (see \cite{MR499562}, 17.4) to see that it is enough to show that the seminorms of the form
$$p_{n_3,n_2,n_1}(F):=\|\partial_3^{n_3}\partial_2^{n_2}\partial_1^{n_1}(F)\|_{\infty}, \; \; F\in\mc{D}$$
yield polynomially increasing sequences on $\xi_n$. We show that there is a constant $C$ (which depends on $n_1, n_2$ and $n_3$ but not on $n$) such that 
\begin{equation}
\label{inequ}
p_{n_3,n_2,n_1}(\xi_n)\leqslant C(1+n)^{n_3+n_2}.
\end{equation}
Trivializing over $U$ or $V$, $\xi_n$ restricts either to $\chi_1$ or $\chi_2$ (compare Lemma \ref{Lem:MorEqQHM}). We estimate over $U$ in the first case:
\begin{align*}
\|p_{n_3,n_2,n_1}\xi_n\|=\|\partial_3^{n_3}\partial_2^{n_2}\partial_1^{n_1}\chi_1\|&\leqslant Cn^{n_3} \left\|\partial_2^{n_2}\chi_1^{(n_1)} \right\| \\
&\leqslant C' n^{n_3} \left\|( n x)^{n_2} \chi_1^{(n_1)} \right\|,
\end{align*}
since $\chi_1^{(n_1)}$ doesn't depend on $y$. The case of $\chi_2$ is treated similarly.
\end{proof}
Hence, applying Theorem \ref{generalth} to $kk$ and $HP$, we get
\begin{Th}\label{QHMDIAG} There is a commutative diagram
\[\scalebox{0.85}{\xymatrix{K_0(C^\infty(\mT^2))\ar[rr]\ar[dr]&&K_0(C^\infty(\mT^2))\ar[rr]\ar[d]&&K_0(\mc{D})\ar[ddd]\ar[dl]\\
&HP_0(C^\infty(\mT^2))\ar[r]&HP_0(C^\infty(\mT^2))\ar[r]&HP_0(\mc{D})\ar[d]\\
&HP_1(\mc{D})\ar[u]&HP_1(C^\infty(\mT^2))\ar[l]&HP_1(C^\infty(\mT^2))\ar[l]\\
K_0(\mc{D})\ar[uuu]\ar[ur]&&K_1(C^\infty(\mT^2))\ar[ll]\ar[u]&&K_1(C^\infty(\mT^2))\ar[ul]\ar[ll]\\
}}\]
\end{Th}

\bibliographystyle{alpha}
\bibliography{BiblioLoc}
\end{document}